\documentclass[12pt]{amsart}

\newtheorem{theorem}{Theorem}[section]
\newtheorem{lemma}[theorem]{Lemma}
\newtheorem{proposition}[theorem]{Proposition}
\newtheorem{definition}[theorem]{Definition}
\newtheorem{remark}[theorem]{Remark}

\numberwithin{equation}{section}

\begin{document}
\baselineskip=16pt

\title[Real projective structures on a real curve]{Real
projective structures on a real curve}

\author[I. Biswas]{Indranil Biswas}

\address{School of Mathematics, Tata Institute of Fundamental Research,
Homi Bhabha Road, Bombay 400005, India}

\email{indranil@math.tifr.res.in}

\author[J. Hurtubise]{Jacques Hurtubise}

\address{Department of Mathematics, McGill University, Burnside
Hall, 805 Sherbrooke St. W., Montreal, Que. H3A 2K6, Canada}

\email{jacques.hurtubise@mcgill.ca}



\date{}

\begin{abstract}
Given a compact connected Riemann surface $X$ equipped with an
antiholomorphic involution $\tau$, we consider the projective 
structures on $X$ satisfying a compatibility condition with 
respect to $\tau$. For a projective structure $P$ on $X$, there 
are holomorphic connections and holomorphic differential 
operators on $X$ that are constructed using $P$. When the 
projective 
structure $P$ is compatible 
with $\tau$, the relationships between $\tau$ and the holomorphic 
connections, or the differential operators, associated to $P$
are investigated. The moduli space of
projective structures on a compact oriented $C^\infty$ surface
of genus $g\, \geq\, 2$ has a natural holomorphic symplectic 
structure. It is known that
this holomorphic symplectic manifold is isomorphic to the
holomorphic symplectic manifold defined by the total 
space of the holomorphic cotangent bundle of the Teichm\"uller
space ${\mathcal T}_g$ equipped with the Liouville symplectic 
form. We show that there is an isomorphism between these two 
holomorphic symplectic manifolds that is compatible with $\tau$.
\end{abstract}

\maketitle

\section{Introduction}

A projective structure on a compact Riemann surface $X$ is
defined by giving
a covering of $X$ by holomorphic coordinate charts such that
all the transition functions are M\"obius transformations.
Projective structures have other equivalent formulations
using projective connections, differential operators etc.

Assume that there is an antiholomorphic involution
$$
\tau\, :\, X\, \longrightarrow\, X\, .
$$
Just as compact Riemann surfaces are same as smooth
projective curves defined over $\mathbb C$, pairs of the form
$(X\, ,\tau)$ are same as geometrically irreducible smooth
projective curves over $\mathbb R$. We consider projective
structures on $X$ compatible with $\tau$; a projective
structure $P$ on $X$ is compatible with $\tau$ if $\tau$
takes any holomorphic coordinate function associated to $P$ to
the conjugate of another holomorphic coordinate function 
associated to $P$. So a projective
structure on $X$ compatible with $\tau$ can be called
a \textit{real} projective structure.

This ``reality'' of projective structures becomes more clear
if we consider the equivalent formulations using
projective connections or differential operators. Associated
to each projective structure is a holomorphic (equivalently,
algebraic) projective connection. A projective structure
is real if and only if the corresponding projective connection
is defined over $\mathbb R$. Also, associated
to each projective structure is a holomorphic (equivalently,
algebraic) differential operator of order three from
$TX$ to $((TX)^*)^{\otimes 2}$, where $TX$ is the holomorphic
tangent bundle. A projective structure
is real if and only if the corresponding differential operator
is defined over $\mathbb R$. We investigate these interrelations.

To define a projective structure on a
$C^\infty$ oriented surface, we do not need
to fix a complex structure on the surface. On the contrary, given
a projective structure, there is a underlying complex structure on
the surface.

Let $Y$ be a compact connected oriented $C^\infty$ surface.
Let $\text{Diffeo}_0(Y)$ denote the group of diffeomorphisms of
$Y$ homotopic to the identity map. This group acts on the
space of all projective structures on $Y$ compatible with the
orientation of $Y$. The corresponding quotient space will
be denoted by ${\mathcal P}_0(Y)$, which is a complex manifold
equipped with a natural holomorphic symplectic form. The
Teichm\"uller space
${\mathcal T}(Y)$ for $Y$ is the quotient by $\text{Diffeo}_0(Y)$
of the space of all complex structures on $Y$ compatible with
its orientation. There is a natural projection of
$$
\varphi\, :\,{\mathcal P}_0(Y)\,\longrightarrow\,
{\mathcal T}(Y)
$$
making ${\mathcal P}_0(Y)$ a torsor for the holomorphic cotangent
bundle $\Omega^1_{{\mathcal T}(Y)}\,\longrightarrow\,
{\mathcal T}(Y)$. Let
$$
\tau\, :\, Y\, \longrightarrow\, Y
$$
be an orientation reversing diffeomorphism of order two. Let
$\tau_P$ (respectively, $\tau_T$) be the involution of
${\mathcal P}_0(Y)$ (respectively, ${\mathcal T}(Y)$) constructed
using $\tau$. We prove that there is a holomorphic section of
the projection $\varphi$ that
\begin{itemize}
\item intertwines $\tau_P$ and
$\tau_T$, and

\item the corresponding biholomorphism of
${\mathcal P}_0(Y)$ with $\Omega^1_{{\mathcal T}(Y)}$ takes
the natural symplectic form on ${\mathcal P}_0(Y)$ to the
Liouville symplectic form on $\Omega^1_{{\mathcal T}(Y)}$.
\end{itemize}

In a work with Huisman, \cite{BHH}, the representations
associated to stable real vector bundles are introduced.
The $\text{PGL}_2$ analog of these representations arise
in real projective structures.

\section{Real projective structures}\label{sec2}

Let $X$ be a compact connected Riemann surface. Let
$$
J\, :\,
T^{\mathbb R}X\, \longrightarrow\, T^{\mathbb R}X
$$
be the almost complex structure of $X$.
Assume that $X$
is equipped with an anti--holomorphic involution
\begin{equation}\label{tau}
\tau\, :\, X\, \longrightarrow\, X\, .
\end{equation}
This means that $\tau$ is a smooth self--map of $X$ such
that $\tau\circ\tau\, =\, \text{Id}_X$, and the equality
\begin{equation}\label{J}
{\rm d}\tau(J(x)(v)) \, =\, -J(\tau(x))({\rm d}\tau(v))
\end{equation}
holds for all $x\, \in\, X$ and $v\, \in\, T^{\mathbb R}_xX$,
where ${\rm d}\tau\, :\, T^{\mathbb R}X\, \longrightarrow\,
T^{\mathbb R}X$ is the differential of $\tau$. As mentioned
in the introduction, such a pair $(X\, ,\tau)$ corresponds to a
geometrically irreducible smooth projective curve defined over
$\mathbb R$ (See \cite{GI}, \cite{Si}, \cite{GH}, \cite{Mi} for 
curves defined over $\mathbb R$.)

We will define projective structures on $X$ compatible with
$\tau$. Before that, let us recall the definition of a projective
structure.

The standard action of $\text{SL}(2,{\mathbb C})$ on ${\mathbb C}^2$
induces an action of the M\"obius group
$\text{PSL}(2,{\mathbb C})$ on ${\mathbb C}{\mathbb P}^1$, and
furthermore, $\text{PSL}(2,{\mathbb C})\, =\, \text{Aut}(
{\mathbb C}{\mathbb P}^1)$ (the group of
holomorphic automorphisms). A \textit{holomorphic 
coordinate
function} on an open subset $U\, \subset\, X$ is an injective 
holomorphic
map $\phi\, :\, U\,\longrightarrow\, {\mathbb C}{\mathbb P}^1$.

A projective structure on $X$ is defined by giving a collection
$\{(U_i\, , \phi_i)\}_{i\in I}$ of holomorphic coordinate
functions such that
\begin{itemize}
\item $\bigcup_{i\in I} U_i\, =\, X$, and

\item for each pair $i\, ,i'\,\in\, I$, the composition
\[
\phi_{i'}\circ \phi^{-1}_i\, :\, \phi_i(U_i\cap U_{i'})
\, \longrightarrow\, \phi_{i'}(U_i\cap U_{i'})
\]
is the restriction of some M\"obius transformation.
\end{itemize}

Two such data $\{(U_i\, , \phi_i)\}_{i\in I}$ and
$\{(U_i\, , \phi_i)\}_{i\in I'}$ satisfying the above
conditions are called {\it equivalent} if their union
$\{U_i\, , \phi_i\}_{i\in I\cup I'}$ also satisfies the
two conditions. A \textit{projective structure} on
$X$ is an equivalence class of such data.
Given a projective structure on $X$, the
holomorphic coordinate functions compatible with it will
be called \textit{projective coordinates}.

Consider the involution $\tau$ in Eq. \eqref{tau}. Note that
for any holomorphic coordinate function $(U\, , \phi)$,
the composition $(\tau^{-1}(U)\, , \overline{\phi\circ\tau})$ 
is also a holomorphic coordinate function.

\begin{definition}\label{def1}
{\rm A projective structure on $X$ is said to be {\it compatible}
with $\tau$ if for each projective coordinate $(U\, ,\phi)$,
the composition $(\tau^{-1}(U)\, , \overline{\phi\circ\tau})$ is
also a projective coordinate.}

{\rm For convenience, a projective structure on $X$ compatible
with $\tau$ will also be called a {\it real} projective structure.}
\end{definition}

\begin{remark}\label{rem1}
{\rm The complex projective line ${\mathbb C}{\mathbb P}^1$ has the
real structure defined by $z\, \longmapsto\, \overline{z}$.
The condition that a projective structure is real means that
the projective structure is compatible with $\tau$ and this
real structure on ${\mathbb C}{\mathbb P}^1$. We note that
${\mathbb C}{\mathbb P}^1$ has another real structure defined
by $z\, \longmapsto\, -1/\overline{z}$. However,
the real projective structures on $X$ defined using this real structure
are not different from those defined above because the two real 
structures on ${\mathbb C}{\mathbb P}^1$ differ by an element of the
M\"obius group.}
\end{remark}

Given a projective structure $P$ on $X$, we can construct another
projective structure $P'$ on $X$ which is uniquely determined
by the following condition. Take any holomorphic
coordinate function $(U\, , \phi)$ compatible with $P$, then
the composition $(\tau(U)\, , \overline{\phi\circ\tau})$ is
compatible with $P'$.

Clearly this construction defines an involution on the
space of all projective structures on $X$.
The following lemma is obvious.

\begin{lemma}\label{le1}
A projective structure $P$ on $X$ is real if and only if
$P$ is fixed by the involution constructed above.
\end{lemma}

We will describe the space of all projective structures on $X$
compatible with $\tau$. We begin with the following simple lemma.

\begin{lemma}\label{lem1}
There exists a projective structure on $X$ compatible with $\tau$.
\end{lemma}

\begin{proof}
If $X\,=\, {\mathbb C}{\mathbb P}^1$,
the unique projective structure on ${\mathbb C}{\mathbb P}^1$,
which is defined by taking the identity map of
${\mathbb C}{\mathbb P}^1$ to be a projective coordinate, is
clearly compatible with $\tau$.

Assume that $\text{genus}(X)\, \geq\, 1$. Let
$$
\gamma\, :\, {\widetilde X}\, \longrightarrow\, X
$$
be a universal cover of $X$, where ${\widetilde X}$ is either
the complex line $\mathbb C$ or the upper half plane
${\mathbb H}\, =\, \{z\in {\mathbb C}\mid
\text{Im}(z)>0\}$. Fix a biholomorphism of $\widetilde X$
with $\mathbb C$ or $\mathbb H$. The holomorphic coordinates
on $X$ given by the inclusion of ${\widetilde X}$ in ${\mathbb 
C}{\mathbb P}^1$ define a projective structure on $X$,
which we will denote by $P$. We will
show that $P$ is compatible with $\tau$ Eq. \eqref{tau}.

The involution $\tau$ lifts to an anti--holomorphic automorphism of 
${\widetilde X}$. Any lift
$$
\widetilde{\tau}\, :\, {\widetilde X}\, \longrightarrow\,
{\widetilde X}
$$
of $\tau$ is of the form $\widetilde{\tau}(z) \, =\,
\overline{T(z)}$ (respectively, $\widetilde{\tau}(z) \, =\,
-\overline{T(z)}$) if ${\widetilde X}\,=\, \mathbb C$
(respectively, ${\widetilde X}\,=\,\mathbb H$),
where $T$ is a holomorphic automorphism
of ${\widetilde X}$. Since ${\widetilde X}$ is either
$\mathbb C$ or $\mathbb H$, the map defined by
$z\, \longmapsto\, T(z)$ is a M\"obius transformation. 
Hence the map defined by
$z\, \longmapsto\, \overline{\widetilde{\tau}(z)}$
is a M\"obius transformation. Therefore, the projective structure
$P$ is compatible with $\tau$.
This completes the proof of the lemma.
\end{proof}

Let $V$ be a holomorphic vector bundle over $X$. Let $\overline{V}$
denote the $C^\infty$
complex vector bundle over $X$ which is identified
with $V$ as a real $C^\infty$ vector bundle, while the complex
structure of the fiber $\overline{V}_x$, $x\,\in\, X$, is the conjugate
of the complex structure of $V_x$. The smooth complex vector bundle
$\tau^*\overline{V}$ has a natural holomorphic structure. A $C^\infty$
section of $\tau^*\overline{V}$ defined over an open subset $U\,
\subset\, X$ is holomorphic if the corresponding section of $V$
over $\tau(U)$ is holomorphic; it is easy to check that this
condition defines a holomorphic structure on $\tau^*\overline{V}$.

Let $TX$ denote the holomorphic tangent bundle of $X$. Since
the differential ${\rm d}\tau$ takes the almost complex structure
$J$ to $-J$ (see Eq. \eqref{J}), it follows immediately that
${\rm d}\tau$ gives a $C^\infty$ isomorphism of $TX$ with
$\tau^*\overline{TX}$. It is easy to check that this is a
holomorphic isomorphism. Consequently, $(TX)^{\otimes m}\, =\,
\tau^*\overline{(TX)^{\otimes m}}$ for all $m\,\in\,{\mathbb Z}$.

Fix a holomorphic line bundle $L\, \longrightarrow\, X$
such that $L^{\otimes 2}$ is holomorphically
isomorphic to $TX$. So, $L$ is a theta characteristic on $X$.
Fix an isomorphism of
$L^{\otimes 2}$ with $TX$. We have the following short exact
sequence of vector bundles on $X$
\begin{equation}\label{e1}
0\, \longrightarrow\, L^*\,=\, L\otimes K_X \, \longrightarrow
\, J^1(L) \, \longrightarrow\, L\, \longrightarrow\,0\, ,
\end{equation}
where $J^1(L)$ is the jet bundle, and $K_X$ is the
holomorphic cotangent bundle. Let
\begin{equation}\label{e2}
{\mathbb P}_L\,\longrightarrow\, X
\end{equation}
be the principal $\text{PGL}(2,{\mathbb C})$--bundle defined by
the projective bundle ${\mathbb P}(J^1(L))$.

If $\text{genus}(X)\, >\, 1$, then $J^1(L)$ is indecomposable. If
$\text{genus}(X)\, =\, 1$, then $J^1(L)\,=\, L\oplus L^*$, and if
$\text{genus}(X)\, =\,0$, then $J^1(L)$ is a trivial vector bundle.
Hence by a criterion of Atiyah and Weil, \cite{At1}, \cite{We}, the
vector bundle $J^1(L)$ in Eq. \eqref{e2} admits a holomorphic
connection. Consequently,
the projective bundle ${\mathbb P}_L$ admits a
holomorphic connection. For a holomorphic connection $\nabla$
on ${\mathbb P}_L$, the second fundamental form of the
holomorphic section
of ${\mathbb P}_L$ defined by the subbundle $L\otimes K_X$
in Eq. \eqref{e1} is a section
\begin{equation}\label{e3}
S(\nabla, L\otimes K_X)\, \in\, H^0(X,\, \textit{H}om
(L\otimes K_X,L)\otimes K_X)\,=\, H^0(X,\,{\mathcal O}_X)\, .
\end{equation}

We recall that the projective structures on $X$ are in bijective
correspondence with the holomorphic connections $\nabla$ on
${\mathbb P}_L$ such that the second fundamental form $S(\nabla, 
L\otimes K_X)$ in Eq. \eqref{e3} is the constant function $1$
(see \cite{Gu}).

\begin{lemma}\label{lem2}
The involution $\tau$ in Eq. \eqref{tau} has a canonical lift to
a $C^\infty$ involution of ${\mathbb P}_L$.

A projective structure on $X$ is real if and only if the 
corresponding holomorphic connection on ${\mathbb P}_L$
is preserved by the above involution of ${\mathbb P}_L$.
\end{lemma}

\begin{proof}
Let $\xi\, \longrightarrow\, X$ be a holomorphic line
bundle. Let $\xi_0\, \longrightarrow\, X$ be a holomorphic line
bundle of finite order. We will show that there is a canonical
isomorphism
\begin{equation}\label{e4}
J^m(\xi\otimes\xi_0)\, =\, J^m(\xi)\otimes\xi_0
\end{equation}
for all $m\, \geq\, 0$.

Let $k$ be a positive integer such that $\xi^{\otimes k}_0$ is 
holomorphically trivial. There is a unique connection $\nabla_0$ on
$\xi_0$ such that the connection on $\xi^{\otimes k}_0
\,=\, {\mathcal O}_X$ induced by $\nabla_0$
is the trivial one (it has trivial monodromy).
This connection $\nabla_0$ is independent of the choice of $k$.
Take any $v\, \in\, J^m(\xi)_{x}$ and $\alpha\, \in\, (\xi_0)_x$,
where $x\, \in\, X$. Take a holomorphic section $\widetilde{v}$
of $\xi$ defined around $x$ such that $v$ represents $\widetilde{v}$.
Let $\widetilde{\alpha}$ be the unique flat section, with respect to
the connection $\nabla_0$, of $\xi_0$ defined around $x$ such that
$\widetilde{\alpha}(x) \,=\, \alpha$. Now sending $v\otimes\alpha$
to the element of $J^m(\xi\otimes\xi_0)_x$ representing the
section $\widetilde{v}\otimes \widetilde{\alpha}$ we get a
homomorphism from $J^m(\xi)\otimes\xi_0$ to $J^m(\xi\otimes\xi_0)$.
It is easy to see that this homomorphism is a
holomorphic isomorphism.

Let $L\, \longrightarrow\, X$ be a holomorphic line bundle
such that $L^{\otimes 2}$ is isomorphic to $TX$.
Consider the line bundle $L_1\, :=\, \tau^*\overline{L}$.
Since $L^{\otimes 2}_1\, =\, \tau^*\overline{TX} \,=\, TX$,
there is a line bundle $\xi_0$ of order two such that
$L_1\,=\, L\otimes \xi_0$. Hence from Eq. \eqref{e4},
$$
J^1(L_1)\, =\, J^1(L)\otimes \xi_0\, .
$$
Consequently, there is a canonical isomorphism
$$
{\mathbb P}_{L_1}\, :=\, {\mathbb P}(J^1(L_1))\,
\stackrel{\sim}{\longrightarrow}\,{\mathbb P}_L\, .
$$
This isomorphism gives a $C^\infty$ involution
\begin{equation}\label{et1}
\widehat{\tau}\, :\,{\mathbb P}_L\,\longrightarrow\,
{\mathbb P}_L
\end{equation}
that lifts $\tau$.

For any
point $x\, \in\, X$, the isomorphism $\widehat{\tau}(x)$ of
the fiber $({\mathbb P}_L)_x$ with
$({\mathbb P}_L)_{\tau(x)}$ is anti--holomorphic. In fact,
the pulled back fiber bundle
$\tau^*{\mathbb P}_L$ has a natural holomorphic
structure which is uniquely determined by the following
condition: a section of $\tau^*{\mathbb P}_L$ over an 
open subset $U\, \subset\, X$ is holomorphic if and only
if the corresponding section of ${\mathbb P}_L$ over
$\tau(U)$ is holomorphic. The involution $\widehat{\tau}$
in Eq. \eqref{et1} is a holomorphic isomorphism of
${\mathbb P}_L$ with $\tau^*{\mathbb P}_L$ equipped
with the above holomorphic structure.

Let $P$ be a projective structure on $X$.
Let $\nabla$ be the holomorphic connection on ${\mathbb P}_L$
associated to $P$. Note that any holomorphic connection on a
Riemann surface is flat, because there are no nonzero 
holomorphic two--forms on it (the curvature of a holomorphic 
connection is a holomorphic two--form with values in the
adjoint bundle). Let $\widetilde{P}$ be the
projective structure on $X$ given by $P$ using the involution
in Lemma \ref{le1}. The flat connection on
${\mathbb P}_L$ corresponding to $\widetilde{P}$ coincides
with $\widehat{\tau}^*\nabla$, where
$\widehat{\tau}$ is the involution in Eq. \eqref{et1}.
Therefore, Lemma \ref{le1} completes the proof.
\end{proof}

\section{Theta characteristics and projective structure}\label{sec3}

\subsection{First jet bundle of a theta 
characteristic}\label{3.1}

A \textit{theta characteristic} of $(X\, ,\tau)$ is a
holomorphic line bundle $\theta\, \longrightarrow\, X$
such that
\begin{itemize}
\item $\theta^{\otimes 2}\, =\, K_X$, and

\item the line bundle $\tau^*\overline{\theta}$ is
holomorphically isomorphic to $\theta$.
\end{itemize}

A theta characteristic $\theta$ of $(X\, ,\tau)$ is called
\textit{real} if there is a holomorphic isomorphism
$$
f\, :\, \theta\, \longrightarrow\, \tau^*\overline{\theta}
$$
such that the composition
$$
\theta\, \stackrel{f}{\longrightarrow}\, \tau^*\overline{\theta}
\, \stackrel{\tau^*\overline{f}}{\longrightarrow}\, 
\tau^*\overline{\tau^*\overline{\theta}}\, =\, \theta
$$
is the identity map.

A theta characteristic $\theta$ of $(X\, ,\tau)$ is called
\textit{quaternionic} if there is a holomorphic isomorphism
$$
f\, :\, \theta\, \longrightarrow\, \tau^*\overline{\theta}
$$
such that the composition
$$
\theta\, \stackrel{f}{\longrightarrow}\, \tau^*\overline{\theta}
\, \stackrel{\tau^*\overline{f}}{\longrightarrow}\,
\tau^*\overline{\tau^*\overline{\theta}}\, =\, \theta
$$
is $-\text{Id}_\theta$.

The set of theta characteristics on $(X\, ,\tau)$ decomposes 
into a disjoint union of real and quaternionic theta characteristics.
It is known that $(X\, ,\tau)$ admits a theta
characteristic (see \cite[pp. 61--62]{At2}). We note that if
$\tau$ does not have any fixed points, then there are no real 
holomorphic line bundles of odd degree on $X$. Hence in that
case there are no real theta characteristics provided the
genus of $X$ is even; see \cite{BHH} for detailed discussions
on the existence of real and quaternionic vector bundles, in
particular, line bundles.

Let $\theta$ be a theta characteristic on $(X\, ,\tau)$. Fix a
holomorphic isomorphism
\begin{equation}\label{f}
f\, :\, \theta\, \longrightarrow\, \tau^*\overline{\theta}
\end{equation}
such that the composition
$$
\theta\, \stackrel{f}{\longrightarrow}\, \tau^*\overline{\theta}
\, \stackrel{\tau^*\overline{f}}{\longrightarrow}\,
\tau^*\overline{\tau^*\overline{\theta}}\, =\, \theta
$$
is either $\text{Id}_\theta$ or $-\text{Id}_\theta$.

We have $\bigwedge^2 J^1(\theta^*) \, =\, \theta^*\otimes
\theta^*\otimes K_X \,=\, {\mathcal O}_X$ (see Eq. \eqref{e1}).
We noted in Section \ref{sec2} that $J^1(\theta^*)$ admits a
holomorphic connection.

The isomorphism in Eq. \eqref{f} induces an isomorphism
\begin{equation}\label{wf}
\widehat{f}\, :\, J^1(\theta^*)\, \longrightarrow\, 
\tau^*\overline{J^1(\theta^*)}\, .
\end{equation}
We note that $\tau^*\overline{\widehat{f}}\circ \widehat{f}$
is $\text{Id}_{J^1(\theta^*)}$ (respectively, 
$-\text{Id}_{J^1(\theta^*)}$) if
$\tau^*\overline{f}\circ f$ is $\text{Id}_\theta$ 
(respectively, $-\text{Id}_\theta$). A holomorphic connection
$\nabla$ on $J^1(\theta^*)$ induces a holomorphic connection
on $\tau^*\overline{J^1(\theta^*)}$; this induced connection will
be denoted by $\tau^*{\overline{\nabla}}$.

Let $\mathcal C$ denote the space of all holomorphic connections
$\nabla$ on $J^1(\theta^*)$ such that
\begin{itemize}
\item the isomorphism $\widehat{f}$ is Eq. \eqref{wf} takes $\nabla$
to $\tau^*{\overline{\nabla}}$, and

\item the connection on $\bigwedge^2 J^1(\theta^*)$
induced by $\nabla$ has trivial monodromy.
\end{itemize}

\begin{proposition}\label{prop1}
The space of connections $\mathcal C$ defined above is in bijective
correspondence with the space of all projective structures
on $X$ compatible with $\tau$.
\end{proposition}

\begin{proof}
Take any $\nabla\, \in\, {\mathcal C}$. The holomorphic connection 
on ${\mathbb P}(J^1(\theta^*))$ induced by $\nabla$ will be
denoted by $\nabla'$. From Lemma \ref{lem2} it follows that
$\nabla'$ defines a real projective structure on $(X\, ,\tau)$.

Let $P$ be a real projective structure on $(X\, ,\tau)$. Let
$\nabla^P$ be the corresponding holomorphic connection on the
projective bundle ${\mathbb P}(J^1(\theta^*))$. From Lemma \ref{lem2}
we know that $\nabla^P$ is preserved by the
involution of ${\mathbb P}(J^1(\theta^*))$. Since
$\text{Lie}(\text{SL}(2, {\mathbb C}))\, =\,
\text{Lie}(\text{PGL}(2, {\mathbb C}))$, and $\bigwedge^2J^1
(\theta^*)\,=\, {\mathcal O}_X$, the connection
$\nabla^P$ defines a holomorphic connection on
$J^1(\theta^*)$ such that the induced connection on $\bigwedge^2 
J^1(\theta^*)$ has trivial monodromy.
\end{proof}

See \cite{Ty} for a detailed study of projective structures.

\subsection{Trivializations on nonreduced diagonal}

Let $X$ be a compact connected Riemann surface. Let
$$
\Delta\, :=\, \{(x\, ,x)\, \mid \, x\,\in\, X\}\, \subset\, 
X\times X
$$
be the (reduced) diagonal. For any integer $n\, \geq\, 2$, the
nonreduced diagonal with multiplicity $n$ will be denoted by
$n\Delta$. Consider the holomorphic line bundle
$$
{\mathcal L}\,:=\, p^*_1 K_X \otimes p^*_2 K_X\otimes {\mathcal 
O}_{X\times X}(2\Delta)
$$
on $X\times X$, where $p_i$, $i\, =\, 1\, , 2$, is the projection 
of $X\times X$ to the $i$--th factor.

The Poincar\'e adjunction formula identifies the holomorphic 
tangent bundle $TX$ with the restriction of the line 
bundle ${\mathcal O}_{X\times X}(\Delta)$ to $\Delta$ (the
Riemann surface $X$ is identified with $\Delta$ by sending any
$x\, \in\, X$ to $(x\, ,x)\, \in\, \Delta$). In view of
this identification of $TX$ with
${\mathcal O}_{X\times X}(\Delta)\vert_\Delta$, the restriction
${\mathcal L}\vert_\Delta$ gets identified with the trivial line 
bundle ${\mathcal O}_\Delta$ over $\Delta$. Let
\begin{equation}\label{ps0}
\psi_0\, \in\, H^0(\Delta,\, ,{\mathcal L}\vert_\Delta)
\end{equation}
be the section given by the constant function $1$ through the
identification of ${\mathcal L}\vert_\Delta$ with ${\mathcal 
O}_\Delta$.

Consider the holomorphic involution $\sigma$ of $X\times X$ 
defined by $(x\, ,y)\, \longmapsto\, (y\, ,x)$. This involution 
clearly has a natural lift to an involution of the line bundle 
$\mathcal L$. Let
$$
\widetilde{\sigma}\, :\, {\mathcal L}\, \longrightarrow\,
{\mathcal L}
$$
be this involution over $\sigma$.

The following two statements hold:
\begin{itemize}
\item There is a unique section
\begin{equation}\label{ps1}
\psi_1\, \in\, H^0(2\Delta ,\, {\mathcal
L}\vert_{2\Delta})
\end{equation}
invariant under $\widetilde{\sigma}$ such that the restriction of
$\psi_1$ to $\Delta\, \subset\, 2\Delta$ coincides with
$\psi_0$ in Eq. \eqref{ps0}.

\item The space of all projective structures on $X$ is in 
bijective correspondence with the space of sections
$$
\psi\, \in\, H^0(3\Delta ,\, {\mathcal L}\vert_{3\Delta})
$$
satisfying the condition that the restriction of
$\psi$ to $2\Delta\, \subset\, 3\Delta$ coincides with the above 
section $\psi_1$.
\end{itemize}
(See \cite{BR} for the details.)

Now let $\tau$ be an antiholomorphic involution of $X$. It 
produces a holomorphic isomorphism
$$
\tau''\, :\, K_X\, \longrightarrow\, \tau^*\overline{K}_X
$$
that sends any locally defined holomorphic one--form $w$
to $\sigma^*\omega$ (note that the line bundle 
$\tau^*\overline{K}_X$ has a natural holomorphic structure).
Since $\overline{K}_X$ is identified, as a $C^\infty$ line 
bundle, with $K_X$, this isomorphism $\tau''$ produces a 
$C^\infty$ isomorphism
$$
\tau'\, :\, K_X\, \longrightarrow\, X_K
$$
over $\tau$. The isomorphism $\tau'$ is fiberwise conjugate 
linear. This $\tau'$ produces a a $C^\infty$ isomorphism
$$
\tau_2\, :\, {\mathcal L} \longrightarrow\,{\mathcal L} 
$$
over the involution $\tau\times \tau$ of $X\times X$.
The isomorphism $\tau_2$ is fiberwise conjugate linear. This
$\tau_2$ is given by the holomorphic isomorphism of line
bundles
$$
p^*_1K_X \otimes p^*_2K_X\,
\stackrel{p^*_1\tau''\otimes p^*_2\tau''}{\longrightarrow}
\, p^*_1\tau^*\overline{K}_X\otimes p^*_2\tau^*\overline{K}_X
\,=\, (\tau\times \tau)^*(p^*_1\overline{K}_X\otimes 
p^*_2\overline{K}_X)\, .
$$

Let $P$ be a projective structure on $X$ given by a section
$$
\psi\, \in\, H^0(3\Delta ,\, {\mathcal L}\vert_{3\Delta})
$$
satisfying the condition that the restriction of
$\psi$ to $2\Delta\, \subset\, 3\Delta$ coincides with the
section $\psi_1$ in Eq. \eqref{ps1}. The projective structure $P$ 
is compatible with $\tau$ if and only if
$$
\tau_2(\psi)\,=\, \psi\, .
$$
This is straight--forward to check.

\subsection{Differential operators of order two}

As before, $\theta$ is a theta characteristic of $(X\, ,\tau)$.
Let $J^2(\theta^*)\, \longrightarrow\, X$ be the second order jet
bundle. It fits in the following short exact sequence of
holomorphic vector bundles on $X$:
\begin{equation}\label{Je}
0\, \longrightarrow\, K^{\otimes 2}_X\otimes
\theta^* \, =\, \theta^{\otimes 3} \, 
\stackrel{\iota}{\longrightarrow}\,
J^2(\theta^*)\, \stackrel{q}{\longrightarrow}\, J^1(\theta^*)\,
\longrightarrow\, 0\, .
\end{equation}
A holomorphic differential operator $D$
of order two from $\theta^*$ to $K^{\otimes 2}_X\otimes
\theta^*$ is, by definition, an ${\mathcal O}_X$--linear
homomorphism from
$J^2(\theta^*)$ to $K^{\otimes 2}_X\otimes\theta^*$. In other words,
it is a holomorphic section
\begin{equation}\label{D}
D\, \in\, H^0(X,\, K^{\otimes 2}_X\otimes\theta^*\otimes 
J^2(\theta^*)^*)\, .
\end{equation}
Since $D\circ\iota$ is a holomorphic endomorphism of the line bundle
$K^{\otimes 2}_X\otimes\theta^*$, where $\iota$ is the homomorphism in 
Eq. \eqref{Je},
it follows that $D\circ\iota$ is multiplication by a scalar
$\sigma(D)\, \in\, \mathbb C$. This complex number $\sigma(D)$
is called the \textit{symbol} of $D$. So the symbol is a homomorphism
\begin{equation}\label{symb}
\sigma\, :\, H^0(X,\, \text{Diff}^2(\theta^*\, ,K^{\otimes 
2}_X\otimes\theta^*))
\,=\, H^0(X,\, K^{\otimes 2}_X\otimes\theta^*\otimes J^2(\theta^*)^*)
\,\longrightarrow\,\mathbb C\, ,
\end{equation}
where $\text{Diff}^2(\theta^*\, ,K^{\otimes 2}_X\otimes\theta^*)
\,=\, {\mathcal H}om(J^2(\theta^*)\, ,K^{\otimes 2}_X\otimes\theta^*)$ 
is the holomorphic vector bundle on $X$ associated to the
sheaf of differential operators of order two from $\theta^*$ to
$K^{\otimes 2}_X\otimes\theta^*$. Note that any $D$ as in Eq. \eqref{D} 
with $\sigma (D) \, =\, 1$ gives a holomorphic splitting of the short
exact sequence in Eq. \eqref{Je}.

{}From the properties of the jet bundles it follows that
for any holomorphic vector bundle $W$, there is
a natural injective homomorphism
$$
J^{p+q}(W)\,\longrightarrow\, J^p(J^q(W))
$$
for all $p\, ,q\, \geq\, 0$. This fits in a commutative diagram
\begin{equation}\label{dgm}
\begin{matrix}
&&0 && 0\\
&&\Big\downarrow&&\Big\downarrow\\
0& \longrightarrow & K^{\otimes 2}_X\otimes\theta^* &\longrightarrow &
J^2(\theta^*) &\longrightarrow & J^1(\theta^*) &
\longrightarrow & 0\\
&&\Big\downarrow && {~}\,\Big\downarrow\beta && \Vert\\
0& \longrightarrow & K_X\otimes J^1(\theta^*) &\longrightarrow &
J^1(J^1(\theta^*)) &\stackrel{q_0}{\longrightarrow} & J^1(\theta^*)
&\longrightarrow & 0\\
&&\Big\downarrow&&\Big\downarrow\\
&& K_X\otimes\theta^* & = & K_X\otimes\theta^*\\
&&\Big\downarrow&&\Big\downarrow\\
&&0 && 0
\end{matrix}
\end{equation}
where the top horizontal short
exact sequence is the one in Eq. \eqref{Je}, and the
bottom horizontal short exact sequence
is the jet sequence for the vector bundle $J^1(\theta^*)$.
Consider the homomorphism
$$
J^1(J^1(\theta^*))\, \longrightarrow\,J^1(\theta^*)
$$ induced by the projection $J^1(\theta^*)\,\longrightarrow
\, \theta^*$. The vertical
homomorphism $J^1(J^1(\theta^*))\, \longrightarrow\,
K_X\otimes\theta^*$ in Eq. \eqref{dgm} is the difference between this
homomorphism and the projection $J^1(J^1(\theta^*))\, \longrightarrow
\,J^1(\theta^*)$ in Eq. \eqref{dgm}.

Take a holomorphic differential operator
$D$ as in Eq. \eqref{D} such that
$\sigma (D) \, =\, 1$. We noted above that $D$ gives a holomorphic
splitting of Eq. \eqref{Je}. Let
\begin{equation}\label{r.d}
\rho_D\, :\, J^1(\theta^*)\, \longrightarrow\, J^2(\theta^*)
\end{equation}
be the holomorphic homomorphism of vector bundles corresponding to
the splitting of Eq. \eqref{Je} given by $D$. The composition
\begin{equation}\label{ch}
\beta\circ \rho_D\, :\, J^1(\theta^*)\, \longrightarrow\,
J^1(J^1(\theta^*))\, ,
\end{equation}
where $\beta$ is the homomorphism in Eq. \eqref{dgm}, gives
a holomorphic splitting of the bottom short exact sequence in
Eq. \eqref{dgm}. In other words,
$$
q_0\circ \beta\circ \rho_D\,=\, \text{Id}_{J^1(\theta^*)}\, ,
$$
where $q_0$ is the projection in Eq. \eqref{dgm}.

Therefore, $\beta\circ \rho_D$ defines a holomorphic connection
on the vector bundle $J^1(\theta^*)$ (see \cite{At1}). Let
$\nabla(D)$ be this holomorphic connection on $J^1(\theta^*)$
constructed from $D$.

Let
\begin{equation}\label{D2}
{\mathcal D}_X \, \subset\, 
H^0(X,\, \text{Diff}^2(\theta^*\, ,K^{\otimes 2}_X\otimes\theta^*))
\end{equation}
be the space of all second order differential operators
$D$ such that
\begin{itemize}
\item $\sigma(D)\, =\,1$, and
\item the corresponding connection $\nabla(D)$ has the property
that the connection on $\bigwedge^2 J^1(\theta^*)$
induced by $\nabla (D)$ has trivial monodromy.
\end{itemize}

There is a bijective correspondence between the space of
all projective structures on $X$ and ${\mathcal D}_X$
defined in Eq. \eqref{D2}. We recall below the construction
of a differential operator lying in ${\mathcal D}_X$ from
a projective structure on $X$.

Let $P$ be a projective structure on $X$. Let
$$
\phi\, :\, U\, \longrightarrow\, {\mathbb C}\,\subset\,
{\mathbb C}{\mathbb P}^1
$$
be a holomorphic coordinate function compatible with $P$.
Fix a holomorphic section $\omega$ of $\theta^*$ over $U$
such that the section $\omega\otimes\omega$ of
$\theta^*\otimes\theta^*\, =\, TX$ coincides with the vector
field $\frac{\partial}{\partial z}$ on $U$ with defined by
the coordinate function $\phi$. Let $D_U$ be
the second order holomorphic differential operator
$$
\theta^*\vert_U \, \longrightarrow\, (K^{\otimes 
2}_X\otimes\theta^*)\vert_U
$$
defined by
$$
D_U(h\cdot\omega)\, =\, \frac{d^2h}{dz^2}\cdot 
(dz)^{\otimes 2}\otimes\omega\, ,
$$
where $h$ is any holomorphic function on $U$. It is straight--forward
to check that the differential operator $D_U$ is independent of the 
choice of $\phi$ compatible with $P$ (it depends only on $P$).
Therefore, these locally defined differential operators $D_U$
patch together compatibly to define a global differential
operator from $\theta^*$ to $K^{\otimes 2}_X\otimes\theta^*$.

Fix a holomorphic isomorphism $f$ as in Eq. \eqref{f}. Let
$$
f_0\, :\, \theta^*\, \longrightarrow\, \tau^*\overline{\theta^*}
\,=\, \tau^*\overline{\theta}^*
$$
be the isomorphism induced by $f$. This isomorphism $f_0$ induces
an isomorphism
\begin{equation}\label{wf2}
\widehat{f}_2\, :\, J^2(\theta^*)\, \longrightarrow\, 
J^2(\tau^*\overline{\theta^*})\,=\,
\tau^*\overline{J^2(\theta^*)}\, .
\end{equation}
We note that $\tau^*\overline{\widehat{f}_2}\circ \widehat{f}_2$
is $\text{Id}_{J^2(\theta^*)}$ (respectively, 
$-\text{Id}_{J^2(\theta^*)}$) if
$\tau^*\overline{f}\circ f$ is $\text{Id}_\theta$ 
(respectively, $-\text{Id}_\theta$).

The isomorphisms $\widehat{f}_2$ and $f_0$ together
define a conjugate linear automorphism of the vector
space $H^0(X,\, \text{Diff}^2(\theta^*\, ,K^{\otimes 
2}_X\otimes\theta^*))$, which we will describe.

For any homomorphism of vector bundles
$$
F\, :\, J^2(\theta^*)\, \longrightarrow\, K^{\otimes 
2}_X\otimes\theta^*\, =\, \theta^{\otimes 3}\, ,
$$
define $F'\, :=\, \tau^*\overline{f^{\otimes 3}}\circ 
\tau^*\overline{F}\circ\widehat{f}_2$.
It is straight--forward to check that
$$
F'\, :\, J^2(\theta^*)\, \longrightarrow\,\theta^{\otimes 3}
$$
is a holomorphic homomorphism of vector bundles. Hence we have
a conjugate linear homomorphism
\begin{equation}\label{tD}
\tau_D\, :\, H^0(X,\, \text{Diff}^2(\theta^*\, ,K^{\otimes
2}_X\otimes\theta^*))\, \longrightarrow\,
H^0(X,\, \text{Diff}^2(\theta^*\, ,K^{\otimes
2}_X\otimes\theta^*))
\end{equation}
defined by $F\, \longmapsto\, F'$. Since
$\tau^*\overline{\widehat{f}_2}\circ \widehat{f}_2$
and $\tau^*\overline{f}\circ f$ are both 
$\text{Id}$ (respectively, $-\text{Id}$) if
$\tau^*\overline{f}\circ f$ is $\text{Id}$
(respectively, $-\text{Id}$), it follows that
$\tau_D$ in Eq. \eqref{tD} is an involution.

The involution $\tau_D$ clearly preserves the subset
${\mathcal D}_X$ defined in Eq. \eqref{D2}.

\begin{lemma}\label{lem4}
The bijection between the projective structures on $X$ and ${\mathcal 
D}_X$ takes the projective structures on $X$ compatible with $\tau$
surjectively to the fixed point set $({\mathcal
D}_X)^{\tau_D}$.
\end{lemma}

\begin{proof}
The bijection between the projective structures on $X$ and ${\mathcal
D}_X$ takes the involution on the space of projective structures
given by $\tau$ (see Lemma \ref{le1}) to the involution $\tau_D$
of ${\mathcal D}_X$. In view of this, the lemma follows from
Lemma \ref{le1}.
\end{proof}

\section{Projective structure
and representation of fundamental group}

So far we considered projective structures on a fixed Riemann
surface. We will now consider all projective structures without
fixing the underlying complex structure. So the definitions
in Section \ref{sec2} have to be modified accordingly, which
we do below.

Let $Y$ be a compact connected oriented $C^\infty$ surface.

A $C^\infty$ \textit{coordinate function} on $Y$
is a pair $(U\, ,\phi)$, where $U$ is an open subset
of $Y$, and $\phi\, :\, U\, \longrightarrow\,
{\mathbb C}{\mathbb P}^1$ is an orientation
preserving $C^\infty$ embedding.
 
A projective structure on $Y$ is defined by giving a collection
$\{(U_i\, , \phi_i)\}_{i\in I}$ of $C^\infty$ coordinate
functions such that
\begin{itemize}
\item $\bigcup_{i\in I} U_i\, =\, Y$, and

\item for each pair $i,i'\,\in\, I$, the composition
\[
\phi_{i'}\circ \phi^{-1}_i\, :\, \phi_i(U_i\cap U_{i'})
\, \longrightarrow\, \phi_{i'}(U_i\cap U_{i'})
\]
is the restriction of some M\"obius transformation.
\end{itemize}

Two such data $\{(U_i\, , \phi_i)\}_{i\in I}$ and
$\{(U_i\, , \phi_i)\}_{i\in I'}$ satisfying the above
conditions
are called {\it equivalent} if their union
$\{U_i\, , \phi_i\}_{i\in I\cup I'}$ also satisfies the
two conditions. A \textit{projective structure} on
$Y$ is an equivalence class of such data.

Therefore, a projective structure on $Y$ gives a complex
structure on $Y$ and a projective structure on the
corresponding Riemann surface.

Given a projective structure on $Y$, the
coordinate functions compatible with it will
be called \textit{projective coordinates}.

Let $\text{Diffeo}_0(Y)$ denote the group of all diffeomorphisms
of $Y$ homotopic to the identity map. The group $\text{Diffeo}_0(Y)$
has a natural action on the space of all
projective structures on $Y$.

\begin{definition}\label{defP}
{\rm The quotient by $\text{Diffeo}_0(Y)$ of the space of all
projective structures on $Y$ will be denoted by
${\mathcal P}_0(Y)$.}
\end{definition}

The isomorphism classes of flat
principal $\text{PSL}(2,{\mathbb
C})$--bundle on $Y$ are identified with the equivalence classes of
representations
\begin{equation}\label{R}
{\mathcal R}\, :=\,
\text{Hom}(\pi_1(Y)\, , \text{PSL}(2,{\mathbb C}))/
\text{PSL}(2,{\mathbb C})\, ,
\end{equation}
which, using complex structure of the group $\text{PSL}(2,
{\mathbb C})$, has a natural structure of a complex analytic space;
the irreducible representations form an open subset contained in
the smooth locus, and this open subset has a natural holomorphic
symplectic structure \cite{Go}. Note that
the equivalence classes of homomorphisms from
$\pi_1(Y)$ are independent of the choice of the base
point in $Y$ needed to define the fundamental group; hence
we omit the base point from the notation. 

A projective structure $P$ on $Y$ gives a flat
principal $\text{PSL}(2,{\mathbb C})$--bundle on $Y$. We
will briefly recall the construction of this flat principal 
bundle. Take a $\{(U_i\, , \phi_i)\}_{i\in I}$ giving a
projective structure on $Y$.
On each $U_i$, consider the trivial principal 
$\text{PSL}(2,{\mathbb C})$--bundle
$$
E^i_{\text{PSL}(2,{\mathbb C})}\, :=\, U_i\times 
\text{PSL}(2,{\mathbb C})\, .
$$
Note that this trivial bundle has a natural flat connection
given by the constant sections of the principal bundle. For
any ordered pair $i\, ,j \, \in\, I$ such that $U_i\bigcap U_j
\, \not=\, \emptyset$, glue
the two principal $\text{PSL}(2,{\mathbb C})$--bundles 
$E^i_{\text{PSL}(2,{\mathbb 
C})}$ and $E^j_{\text{PSL}(2,{\mathbb C})}$ over the 
open subset $U_i\bigcap U_j$ using the element of
$\text{PSL}(2,{\mathbb C})$ given by $\phi_j\circ \phi^{-1}_i$.
(Recall that $\phi_j\circ \phi^{-1}_i$ is the restriction of
a M\"obius transformation, and the group of M\"obius 
transformations is identified with $\text{PSL}(2,{\mathbb C})$;
hence $\phi_j\circ \phi^{-1}_i$ gives an element of
$\text{PSL}(2,{\mathbb C})$.) This was way we get a
principal $\text{PSL}(2,{\mathbb C})$--bundle on $Y$. Since
the transition functions are constants, the natural connection
on the trivial principal bundles $\{E^i_{\text{PSL}(2,{\mathbb 
C})}\}_{i\in I}$ patch together compatibly to define a
flat connection on the principal $\text{PSL}(2,{\mathbb 
C})$--bundle over $Y$.

Sending a flat $\text{PSL}(2,{\mathbb C})$--connection to its
monodromy homomorphism, we get a bijection between $\mathcal R
R$ (defined in \eqref{R}) and the isomorphism classes of
flat $\text{PSL}(2,{\mathbb C})$--bundles. Consider
${\mathcal P}_0(Y)$ (see Definition \ref{defP}).
Let
\begin{equation}\label{mu}
\mu\, :\, {\mathcal P}_0(Y)\, \longrightarrow\,
{\mathcal R}
\end{equation}
be the map that sends any projective structure to the
monodromy of the corresponding flat $\text{PSL}(2,
{\mathbb C})$ connection. It is known that the map $\mu$
is injective, and furthermore, its image is an open subset of
$\mathcal R$ contained in the locus of irreducible
representations \cite{He}, \cite{Hu}. Therefore, 
${\mathcal P}_0(Y)$ is a complex manifold; the complex
structure is uniquely determined by the condition that
the map $\mu$ is holomorphic.

The holomorphic symplectic form on the locus of irreducible
representations defines a holomorphic symplectic form on
${\mathcal P}_0(Y)$.

Let
\begin{equation}\label{R1}
{\mathcal R}_P\, \subset\, {\mathcal R}
\end{equation}
be the open subset defined by the projective structures on $Y$.

Let
\begin{equation}\label{t}
\tau\, :\, Y\, \longrightarrow\, Y
\end{equation}
be an orientation reversing diffeomorphism such that
$$
\tau\circ\tau\,=\, \text{Id}_Y\, .
$$

Imitating Definition \ref{def1}, we define the following:

\begin{definition}\label{de0}
{\rm A projective structure on $Y$ is said to be {\it compatible}
with $\tau$ if for each projective coordinate $(U\, ,\phi)$,
the composition $(\tau(U)\, , \overline{\phi\circ\tau})$ is
also a projective coordinate.}
\end{definition}

Our aim in this section is to identify the subset of ${\mathcal R}_P$
(see Eq. \eqref{R1}) that corresponds to the projective structures
compatible with $\tau$.

Fix a base point $y_0\, \in\, Y$ such that $\tau(y_0)\, \not=\,
y_0$. We recall from \cite{BHH} an extension of
${\mathbb Z}/2{\mathbb Z}$ by $\pi_1(Y,y_0)$.

Let $\Gamma$ denote the space of all homotopy classes, with fixed
end points, of continuous
paths $\gamma\, :\, [0\, ,1] \, \longrightarrow\, Y$ such that
\begin{itemize}
\item{} $\gamma(0)\, =\, y_0$, and
\item $\gamma(1)\, \in\, \{y_0\, ,\tau(y_0)\}$.
\end{itemize}
So $\Gamma$
is a disjoint union of $\pi_1(Y,y_0)$ and $Path(Y,y_0)$
(the homotopy classes of paths from $y_0$ to $\tau(y_0)$).
We recall below the group structure of $\Gamma$. For
$\gamma_1\, \in\,\pi_1(Y,y_0)$, we have the
usual
composition $\gamma_2\gamma_1\, =\, \gamma_2\circ \gamma_1$
of paths. If $\gamma_1\, \in\, Path(Y,y_0)$, then define
$$
\gamma_2\gamma_1\, :=\, \sigma(\gamma_2)\circ \gamma_1\, .
$$
Therefore, we have a short exact sequence of groups
\begin{equation}\label{sg}
e\, \longrightarrow\,\pi_1(Y,y_0)\, \longrightarrow\,
\Gamma\, \longrightarrow\, {\mathbb Z}/2{\mathbb Z}
\, \longrightarrow\, e\, .
\end{equation}

Let $\mathcal G$ denote the set of all diffeomorphisms
of ${\mathbb C}{\mathbb P}^1$ that are either holomorphic
or anti--holomorphic. Note that $\mathcal G$ is a group
under the composition of maps. Therefore, we have a short
exact sequence of groups
\begin{equation}\label{sg2}
e\, \longrightarrow\,\text{PSL}(2,{\mathbb C})\, \longrightarrow
\,{\mathcal G}\, \longrightarrow\, {\mathbb Z}/2{\mathbb Z}
\, \longrightarrow\, e\, .
\end{equation}

\begin{proposition}\label{prop2}
The subset of ${\mathcal R}_P$ defined by
the projective structures on $Y$ compatible with $\tau$
consists of those homomorphisms
$$
\rho\, :\, \pi_1(Y,y_0)\, \longrightarrow\,
{\rm PSL}(2,{\mathbb C})
$$
in ${\mathcal R}_P$ (see Eq. \eqref{R1}) that extend to a
homomorphism
$$
\widetilde{\rho}\, :\, \Gamma\, \longrightarrow\,
{\mathcal G}
$$
fitting in the following commutative diagram
$$
\begin{matrix}
e & \longrightarrow &\pi_1(Y,y_0)& \longrightarrow &
\Gamma & \longrightarrow & {\mathbb Z}/2{\mathbb Z}
& \longrightarrow & e\\
&&~\Big\downarrow \rho &&~\Big\downarrow\widetilde{\rho}
&& \Vert\\
e & \longrightarrow &{\rm PSL}(2,{\mathbb C})& \longrightarrow
&{\mathcal G}& \longrightarrow & {\mathbb Z}/2{\mathbb Z}
& \longrightarrow & e
\end{matrix}
$$
(see Eq. \eqref{sg} and Eq. \eqref{sg2}).
\end{proposition}

\begin{proof}
Let $P$ be a projective structure on $Y$ compatible with $\tau$.
Let
$$
({\mathbb P}_P\, ,\nabla)\, \longrightarrow\, Y
$$
be the flat projective bundle of relative (complex)
dimension one associated to $P$. 
We also have a lift of the automorphism
$\tau$ to a $C^\infty$ involution
\begin{equation}\label{tw2}
\widehat{\tau}\,:\, \, {\mathbb P}_P\longrightarrow\,
{\mathbb P}_P
\end{equation}
that preserves the connection $\nabla$
(see Lemma \ref{lem2} and Eq. \eqref{et1}).

Fix a holomorphic isomorphism
\begin{equation}\label{z0}
\zeta_0\, :\, {\mathbb C}{\mathbb P}^1\,
\longrightarrow\, ({\mathbb P}_P)_{y_0}
\end{equation}
of the fiber $({\mathbb P}_P)_{y_0}$ with
the projective line.

For any $\gamma\,\in\,\pi_1(Y,y_0)$, consider the parallel
translation
$$
T(\nabla, \gamma)\, :\,({\mathbb P}_P)_{y_0}\,
\longrightarrow\,({\mathbb P}_P)_{y_0}
$$
along $\gamma$ for the connection $\nabla$.
We have the monodromy homomorphism
for the connection $\nabla$
\begin{equation}\label{r1}
\rho\, :\, \pi_1(Y,y_0)\, \longrightarrow\,
{\rm PSL}(2, {\mathbb C})
\end{equation}
defined by $\gamma\, \longmapsto\,
\zeta^{-1}_0\circ T(\nabla, \gamma)\circ\zeta_0$,
where $\zeta_0$ is the fixed isomorphism in Eq. \eqref{z0}.

Now, take any element $\gamma\,\in\, Path(Y,y_0)$, so
$\gamma$
is a homotopy class of paths from $y_0$ to $\tau(y_0)$.
Let
$$
T(\nabla, \gamma)\, :\,({\mathbb P}_P)_{y_0}\,
\longrightarrow\,({\mathbb P}_P)_{\tau(y_0)}
$$
be the isomorphism of fibers obtained by taking parallel
translation along $\gamma$ for the connection $\nabla$.
Consider the diffeomorphism
\begin{equation}\label{r2}
g_\gamma\, :=\,
\zeta^{-1}_0\circ
\widehat{\tau}_{\tau(y_0)}\circ T(\nabla, \gamma)\circ\zeta_0
\end{equation}
of ${\mathbb C}{\mathbb P}^1$, where
$$
\widehat{\tau}_{\tau(y_0)}\, :\,
({\mathbb P}_P)_{\tau(y_0)}\,\longrightarrow\,
({\mathbb P}_P)_{y_0}
$$
is the restriction of the diffeomorphism
$\widehat{\tau}$ in Eq. \eqref{tw2}. Note that
the diffeomorphism $g_\gamma$ in Eq.
\eqref{r2} is anti--holomorphic,
because $\widehat{\tau}$ is anti--holomorphic, while
$\zeta_0$ and $T(\nabla, \gamma)$ are both holomorphic
isomorphisms.

Let
\begin{equation}\label{r3}
\rho'\, :\, Path(Y,y_0)\, \longrightarrow\,{\mathcal G}
\end{equation}
be the map defined by $\gamma\, \longmapsto\, g_\gamma$,
where $\mathcal G$ is defined in Eq. \eqref{sg2}, and
$g_\gamma$ is constructed in Eq. \eqref{r2}.
The two maps $\rho$ and $\rho'$ constructed in
Eq. \eqref{r1} and Eq. \eqref{r3} respectively
together define a homomorphism
$$
\widetilde{\rho}\, :\, \Gamma\, \longrightarrow\,
{\mathcal G}
$$
that fits in the commutative diagram in the statement
of the proposition.

To prove the converse, let
\begin{equation}\label{r5}
\rho\, :\, \pi_1(Y,y_0)\,\longrightarrow\,
{\rm PSL}(2,{\mathbb C})
\end{equation}
be a homomorphism that lies in 
${\mathcal R}_P$ (defined in Eq. \eqref{R1}).
Let $P$ be the
projective structure on $Y$ defined by $\rho$.

The self--map $\tau$ in Eq. \eqref{t} defines
an isomorphism
$$
\tau_*\, :\, \pi_1(Y,y_0)\,\longrightarrow\,
\pi_1(Y,\tau(y_0))
$$
that sends any loop to its image by $\tau$. Let
\begin{equation}\label{r6}
\overline{\rho}_\tau\, :\, \pi_1(Y,\tau(y_0))
\,\longrightarrow\,{\rm PSL}(2,{\mathbb C})
\end{equation}
be the homomorphism defined by
$g\, \longmapsto\, \overline{\rho(\tau^{-1}_*(g))}$.

Let $\overline{P}$ be the projective structure on $Y$
corresponding to the projective structure $P$ by the
involution in Lemma \ref{le1}. This projective
structure $\overline{P}$ is given by the homomorphism
$\overline{\rho}_\tau$ constructed in Eq. \eqref{r6}.

Assume that there is a homomorphism
\begin{equation}\label{r4}
\widetilde{\rho}\,:\,\Gamma\,\longrightarrow\,{\mathcal G}
\end{equation}
that fits in the commutative diagram
$$
\begin{matrix}
e & \longrightarrow &\pi_1(Y,y_0)& \longrightarrow &
\Gamma & \longrightarrow & {\mathbb Z}/2{\mathbb Z}
& \longrightarrow & e\\
&&~\Big\downarrow \rho &&~\Big\downarrow\widetilde{\rho}
&& \Vert\\
e & \longrightarrow &
{\rm PSL}(2,{\mathbb C})& \longrightarrow
&{\mathcal G}& \longrightarrow & {\mathbb Z}/2{\mathbb Z}
& \longrightarrow & e
\end{matrix}
$$
(as in the statement of the proposition).

Take any element $g_0\, \in\, \mathcal G$ that projects
to the generator of ${\mathbb Z}/2{\mathbb Z}$. Let
$$
\rho(g_0)\, :\, \pi_1(Y,y_0)\, \longrightarrow\,
{\rm PSL}(2,{\mathbb C})
$$
be the homomorphism defined by $\gamma\,\longmapsto\,
\gamma^{-1}_0 \rho(\gamma) g_0$. One can check that
the element in $\mathcal R$ (see Eq. \eqref{R})
given by $\rho(g_0)$ coincides with the element given by
$\overline{\rho}_\tau$ constructed in Eq. \eqref{r6}.

Now from the injectivity of the map $\mu$ in Eq. \eqref{mu}
we conclude that the projective structure $P$ coincides with
the projective structure $\overline{P}$. This completes the
proof of the proposition.
\end{proof}

\section{Projective structure and symplectic form}

We continue with the notation of the previous section. Let
${\mathcal C}(Y)$ denote the space of all complex structures
on $Y$ compatible with the orientation of $Y$. The group
$\text{Diffeo}_0(Y)$ has a natural action on ${\mathcal C}(Y)$.
The quotient
$$
{\mathcal T}(Y)\, :=\, {\mathcal C}(Y)/\text{Diffeo}_0(Y)
$$
is the Teichm\"uller space, which has a natural complex structure.
The involution $\tau$ (see Eq. \eqref{t}) gives an involution
of ${\mathcal T}(Y)$ that sends any almost complex structure
$J$ on $Y$ to $({\rm d}\tau)^{-1}\circ J\circ {\rm d}\tau$, where 
${\rm d}\tau$ is the differential of $\tau$. This involution of
${\mathcal T}(Y)$ is anti--holomorphic.

Consider ${\mathcal P}_0(Y)$ introduced
in Definition \ref{defP}. Let
\begin{equation}\label{psi}
\varphi\, :\, {\mathcal P}_0(Y)\, \longrightarrow\, {\mathcal T}(Y)
\end{equation}
be the forgetful map that sends a projective structure to the
underlying complex structure. It is known that $\varphi$ makes
${\mathcal P}_0(Y)$ a holomorphic fiber bundle over ${\mathcal T}(Y)$.
More precisely, ${\mathcal P}_0(Y)$ is a holomorphic torsor for
the holomorphic cotangent bundle $\Omega^1_{{\mathcal T}(Y)}
\,\longrightarrow\, {\mathcal T}(Y)$;
this means that there is a natural holomorphic map
\begin{equation}\label{cA}
{\mathcal A}\, :\, 
{\mathcal P}_0(Y)\times_{{\mathcal T}(Y)} \Omega^1_{{\mathcal T}(Y)}
\, \longrightarrow\, {\mathcal P}_0(Y)\, ,
\end{equation}
{}from the fiber product, with the property that the
restriction of $\mathcal A$ over any point $t\, \in\, {\mathcal T}(Y)$
is a free transitive action of the cotangent space
$(\Omega^1_{{\mathcal T}(Y)})_t$
on the fiber $({\mathcal P}_0(Y))_t$. Therefore, given any
holomorphic section of the map $\varphi$
\begin{equation}\label{S}
s\, :\, {\mathcal T}(Y)\, \longrightarrow\,
{\mathcal P}_0(Y)\, ,
\end{equation}
we get a holomorphic isomorphism
\begin{equation}\label{cI}
I_s\, :\, \Omega^1_{{\mathcal T}(Y)}\, \longrightarrow\,
{\mathcal P}_0(Y)
\end{equation}
that sends any $v\, \in\, (\Omega^1_{{\mathcal T}(Y)})_t$
to ${\mathcal A}(s(t)\, ,v)$, where ${\mathcal A}$ is the
map in Eq. \eqref{cA}.

Recall that ${\mathcal P}_0(Y)$ has a natural symplectic form which
is obtained by pulling back, by the map $\mu$ in Eq. \eqref{mu},
the natural symplectic form on the smooth locus of $\mathcal R$.
It is known that there are holomorphic sections $s$ as in
Eq. \eqref{S} such that the corresponding map $I_s$ in Eq. \eqref{cI}
takes the Liouville symplectic form on $\Omega^1_{{\mathcal T}(Y)}$
to the natural symplectic form on ${\mathcal P}_0(Y)$
(see \cite{Ka}, \cite{AB}).

Recall that both ${\mathcal T}(Y)$ and ${\mathcal P}_0(Y)$ are
equipped with anti--holomorphic involutions constructed using
$\tau$. Let $\tau_T$ (respectively, $\tau_P$) be the
anti--holomorphic involution of ${\mathcal T}(Y)$
(respectively, ${\mathcal P}_0(Y)$) given by $\tau$.

\begin{lemma}\label{lem6}
There is a holomorphic section
$s\, :\, {\mathcal T}(Y)\, \longrightarrow\,
{\mathcal P}_0(Y)$ of the projection $\varphi$ in
Eq. \eqref{psi} such that
\begin{itemize}
\item $s\circ \tau_T\, =\, \tau_P\circ s$, and

\item the map $I_s$ in Eq. \eqref{cI}
takes the Liouville symplectic form on $\Omega^1_{{\mathcal T}(Y)}$
to the natural symplectic form on ${\mathcal P}_0(Y)$.
\end{itemize}
\end{lemma}

\begin{proof}
Fix a holomorphic section $s_0\, :\, {\mathcal T}(Y)\, \longrightarrow\,
{\mathcal P}_0(Y)$ such that the map $I_{s_0}$ in Eq. \eqref{cI}
takes the Liouville symplectic form on $\Omega^1_{{\mathcal T}(Y)}$
to the natural symplectic form on ${\mathcal P}_0(Y)$; from
\cite{Ka}, \cite{AB} we know that such sections exist. Let
$$
s'_0\, :\, {\mathcal T}(Y)\, \longrightarrow\,
{\mathcal P}_0(Y)
$$
be the section of $\varphi$ defined by $z\, \longmapsto\, 
(\tau_P\circ s_0\circ \tau_T)(z)$. Since both
$\tau_P$ and $\tau_T$ are anti--holomorphic, and $s$ is
holomorphic, it follows that $s'_0$ is holomorphic. Let
$$
\omega\, :\, H^0({\mathcal T}(Y), \Omega^1_{{\mathcal T}(Y)})
$$
be the holomorphic section uniquely determined by the condition
that
$$
{\mathcal A}(s_0(z)\, ,\omega(z))\, =\, s'_0(z)
$$
for all $z\, \in\, {\mathcal T}(Y)$, where $\mathcal A$
is the map in Eq. \eqref{cA}. The involution $\tau_P$ takes
the symplectic form on ${\mathcal P}_0(Y)$ to its negative. From
this it can be deduced that $I_{s'_0}$ also
takes the Liouville symplectic form on $\Omega^1_{{\mathcal T}(Y)}$
to the symplectic form on ${\mathcal P}_0(Y)$. Indeed, it suffices
to show that the image of $s'_0$ is Lagrangian. But this is a 
consequence of the fact that the image of $s_0$ is Lagrangian.

Let
$$
s\, :\, {\mathcal T}(Y)\, \longrightarrow\,
{\mathcal P}_0(Y)
$$
be the holomorphic section of $\varphi$ defined by
$z\, \longmapsto\, {\mathcal A}(s_0(z)\, ,\omega(z)/2)$.
We will show that $s$ satisfies all the conditions in the lemma.

{}From the construction of $s$ it follows that
$s\circ \tau_T\, =\, \tau_P\circ s$.

Using $s_0$, identify ${\mathcal P}_0(Y)$ with 
$\Omega^1_{{\mathcal T}(Y)}$. Using this identification
of ${\mathcal P}_0(Y)$ with $\Omega^1_{{\mathcal T}(Y)}$, both
$s'_0$ and $s$ are holomorphic one--forms on ${\mathcal T}(Y)$.
These one--forms will be denoted by $\widehat{s}'_0$
and $\widehat{s}$ respectively. From the construction
of $s$ it follows that
\begin{equation}\label{12}
\widehat{s}\,=\, \frac{1}{2}\widehat{s}'_0\, .
\end{equation}
Since $I_{s'_0}$ takes the Liouville symplectic form on
$\Omega^1_{{\mathcal T}(Y)}$ to the symplectic form
on ${\mathcal P}_0(Y)$, we conclude that $d\widehat{s}'_0\,=\, 0$.
Hence from Eq. \eqref{12},
$$
d\widehat{s} \, =\, 0\, .
$$
This immediately implies that $I_s$ in Eq. \eqref{cI}
takes the Liouville symplectic form on $\Omega^1_{{\mathcal T}(Y)}$
to the natural symplectic form on ${\mathcal P}_0(Y)$. This
completes the proof of the lemma.
\end{proof}

Let $M$ be a manifold. The total space of the cotangent bundle
$T^*M$ is equipped with the Liouville symplectic form. Let
$S\, \subset\, M$ be a smooth submanifold. Let
$$
N^*_S\, \subset\, (T^*M)\vert_S\, \subset\, T^*M
$$
be the co--normal bundle which is given by the dual of the projection
of $(TM)\vert_S$ to the normal bundle $N_S$ to $S$. The submanifold
$N^*_S\, \subset\, T^*M$ is Lagrangian.

Take any component $\mathcal S$ of the fixed point locus of
the involution $\tau_P$ of ${\mathcal P}_0(Y)$. Fix a section
$s$ as in Lemma \ref{lem6}; identify ${\mathcal P}_0(Y)$ with
$\Omega^1_{{\mathcal T}(Y)}$ using $s$. In terms of this
identification, the submanifold $\mathcal S$ coincides with the
co--normal bundle of the submanifold $\varphi(\mathcal S)\,\subset\,
{\mathcal T}(Y)$.

\section{Symplectic structure on moduli space of 
projective structures}\label{sec-s}

Let $V$ be a complex vector space of dimension two.
For any nonnegative integer $i$, consider the homomorphism
\begin{equation}\label{sle1}
\Phi_i\, :\, H^0({\mathbb P}(V),\, T{\mathbb P}(V))\otimes_{\mathbb C}
{\mathcal O}_{{\mathbb P}(V)}\, \longrightarrow\,
J^i(T{\mathbb P}(V))
\end{equation}
that sends any pair $(x\, ,s)$, where $x\, \in\, {\mathbb P}(V)$
and $s\, \in\, H^0({\mathbb P}(V),\, T{\mathbb P}(V))$, to the
element of $J^2(T{\mathbb P}(V))_x$ obtained by restricting $s$
to the $i$--th order infinitesimal neighborhood of $x$. It is
easy to see that $\Phi_2$ is an isomorphism.
The composition
\begin{equation}\label{sle2}
\Phi_3\circ \Phi^{-1}_2\, :\, J^2(T{\mathbb P}(V))\,
\longrightarrow \,J^3(T{\mathbb P}(V))
\end{equation}
is a splitting of the jet sequence
\begin{equation}\label{sle3}
0\, \longrightarrow\,
K^{\otimes 3}_{{\mathbb P}(V)}\otimes T{\mathbb P}(V)
\,=\,K^{\otimes 2}_{{\mathbb P}(V)}
\,\stackrel{\iota_0}{\longrightarrow}\,
J^3(T{\mathbb P}(V))\,\longrightarrow\,
J^2(T{\mathbb P}(V)) \,\longrightarrow\, 0\, .
\end{equation}
Let
\begin{equation}\label{sle4}
{\mathcal D}({\mathbb P}(V))\, :\, J^3(T{\mathbb P}(V))
\,\longrightarrow\,K^{\otimes 2}_{{\mathbb P}(V)}
\end{equation}
be the unique homomorphism such that
$$
\text{kernel}({\mathcal D}({\mathbb P}(V)))
\,=\, \text{image}(\Phi_3\circ \Phi^{-1}_2)\, ,
$$
and
\begin{equation}\label{sll1}
{\mathcal D}({\mathbb P}(V))\circ\iota_0\, =\,\text{Id}_{
K^{\otimes 2}_{{\mathbb P}(V)}}\, ,
\end{equation}
where $\iota_0$ is the homomorphism
in Eq. \eqref{sle3}. Therefore, ${\mathcal D}({\mathbb P}(V))$
is a global holomorphic differential operator of third order,
more precisely,
$$
{\mathcal D}({\mathbb P}(V))\, \in\, H^0({\mathbb P}(V),\,
\text{Diff}^3(T{\mathbb P}(V), K^{\otimes 2}_{{\mathbb P}(V)}))\, .
$$
The symbol of the differential operator ${\mathcal D}({\mathbb P}(V))$
is a holomorphic section of
$$
(T{\mathbb P}(V))^{\otimes 3}
\otimes {\mathcal H}om(T{\mathbb P}(V),\,
K^{\otimes 2}_{{\mathbb P}(V)})\,=\,{\mathcal O}_X\, ,
$$
and it coincides with ${\mathcal D}({\mathbb P}(V))\circ\iota_0$,
where $\iota_0$ is the homomorphism in Eq. \eqref{sle3}. From
Eq. \eqref{sll1} it follows that the symbol of
the differential operator ${\mathcal D}({\mathbb P}(V))$ is the
constant function $1$.

Let
\begin{equation}\label{sle5}
{\mathcal S}({\mathbb P}(V))\, \longrightarrow\,
{\mathbb P}(V)
\end{equation}
be the local system defined by the sheaf of solutions
of the differential operator
${\mathcal D}({\mathbb P}(V))$; it coincides with 
the trivial vector bundle with fiber
$H^0({\mathbb P}(V),\, T{\mathbb P}(V))\, =\, \text{sl}(V)$,
where $\text{sl}(V)\, \subset\, \text{End}(V)$ is the
endomorphisms of trace zero. Hence
\begin{equation}\label{sle6}
{\mathcal S}({\mathbb P}(V))\,=\, {\mathbb P}(V)\times
\text{sl}(V)\, .
\end{equation}

The standard action of $\text{GL}(V)$ on $V$ defines
an action of $\text{GL}(V)$ on $T({\mathbb P}(V))$; hence we have
an action of $\text{GL}(V)$ on each tensor power of
$T({\mathbb P}(V))$. The differential operator ${\mathcal D}({\mathbb 
P}(V))$ is clearly fixed by the action of $\text{GL}(V)$. Also,
the identification in Eq. \eqref{sle6} is $\text{GL}(V)$--equivariant
(the action of $\text{GL}(V)$ on $\text{sl}(V)$ is the adjoint
one).

Now, let $X$ be a compact connected
Riemann surface equipped with a projective
structure $\mathcal P$. Identifying ${\mathbb C}{\mathbb P}^1$
with ${\mathbb P}(V)$, we may consider the projective coordinates
as embeddings of open subsets of $X$ into ${\mathbb P}(V)$. With
this identification, the transition functions lie in
$\text{PGL}(V)$.

Since the differential operator ${\mathcal D}({\mathbb P}(V))$
in Eq. \eqref{sle4} is $\text{GL}(V)$--invariant, and the
transition functions for the projective coordinate functions lie
in $\text{PGL}(V)$, the projective structure $\mathcal P$ produces
a differential operator
\begin{equation}\label{sle7}
{\mathcal D}_{\mathcal P}\, \in\, H^0(X,\,
\text{Diff}^3(TX, K^{\otimes 2}_X))
\end{equation}
which is constructed locally from ${\mathcal D}({\mathbb P}(V))$.
The symbol of ${\mathcal D}_{\mathcal P}$ is the constant function
$1$ because the symbol of ${\mathcal D}({\mathbb P}(V))$ is so.
Let
\begin{equation}\label{sle8}
{\mathcal S}({\mathcal P})\, \longrightarrow\, X
\end{equation}
be the local system defined by the sheaf of solutions
of ${\mathcal D}_{\mathcal P}$.

Let
$$
{\mathbb P}_{\mathcal P}\, \longrightarrow\, X
$$
be the flat $\text{PGL}(V)$--bundle given by the projective
structure $\mathcal P$ (see Lemma \ref{lem2}); we recall that
the holomorphic $\text{PGL}(V)$--bundle underlying
${\mathbb P}_{\mathcal P}$ coincides with ${\mathbb P}_L$
constructed in Eq. \eqref{e2}. Let
\begin{equation}\label{sle9}
\text{ad}({\mathbb P}_{\mathcal P}) \, \longrightarrow\, X
\end{equation}
be the corresponding flat adjoint vector bundle; recall that
$\text{ad}({\mathbb P}_{\mathcal P})$ is associated to
${\mathbb P}_{\mathcal P}$ for the adjoint action of
$\text{PGL}(V)$ on $\text{sl}(V)$. Let
\begin{equation}\label{sle10}
\underline{\text{ad}({\mathbb P}_{\mathcal P})}
\, \longrightarrow\, X
\end{equation}
be the local system defined by the sheaf of flat sections
of $\text{ad}({\mathbb P}_{\mathcal P})$. Since the identification
in Eq. \eqref{sle6} is $\text{GL}(V)$--invariant, we have
\begin{equation}\label{sle11}
\underline{\text{ad}({\mathbb P}_{\mathcal P})}\,=\,
{\mathcal S}({\mathcal P})\, ,
\end{equation}
where ${\mathcal S}({\mathcal P})$ and $\underline{\text{ad}({\mathbb 
P}_{\mathcal P})}$ are constructed in Eq. \eqref{sle8} and
Eq. \eqref{sle10} respectively.

The space of infinitesimal deformations of a representation
of $\pi_1(X)$ is given by the first cohomology of the local
system defined by the adjoint representation \cite{Go}. In particular,
the space of infinitesimal deformations of the representation
of $\pi_1(X)$
in $\text{PGL}(V)$ associated to $\mathcal P$ is given
by $H^1(X,\, \underline{\text{ad}({\mathbb P}_{\mathcal P})})$.
This implies that the space of infinitesimal deformations of
the projective structure $\mathcal P$ (in the isotopy
classes of projective structures) on the oriented
$C^\infty$ surface $X$ is given by
$H^1(X,\, \underline{\text{ad}({\mathbb P}_{\mathcal P})})$
(recall that the map $\mu$ in Eq. \eqref{mu} is an open embedding).

As in Eq. \eqref{mu}, let ${\mathcal P}_0(X)$ denote the space
of all projective structures on the oriented
$C^\infty$ surface $X$ modulo the group $\text{Diffeo}_0(X)$
of diffeomorphisms of $X$ homotopic to the identity map.
For the projective structure $\mathcal P$
on $X$, we have noted above that
\begin{equation}\label{sle12}
T_{\mathcal P}{\mathcal P}_0(X)\,=\,
H^1(X,\, \underline{\text{ad}({\mathbb P}_{\mathcal P})})\, .
\end{equation}

Let
$$
{\mathcal C}^\bullet \,~\, ~\, {~:~}\,~\, ~\, {\mathcal C}^0\, :=\,
TX \, \stackrel{{\mathcal D}_{\mathcal P}}{\longrightarrow}\,
{\mathcal C}^1\, :=\, K^{\otimes 2}_X
$$
be the complex of sheaves of $X$. From Eq. \eqref{sle11},
Eq. \eqref{sle8} and Eq. \eqref{sle12},
\begin{equation}\label{sle13}
T_{\mathcal P}{\mathcal P}_0(X)\,=\,
{\mathbb H}^1(X,\, {\mathcal C}^\bullet)\, ,
\end{equation}
where ${\mathbb H}^1$ is the hypercohomology.

Consider the tensor product of the complex of
sheaves ${\mathcal C}^\bullet$ with itself
$$
{\mathcal C}^\bullet\otimes_{\mathbb C}
{\mathcal C}^\bullet\, \,~\, ~\, {~:~}\,~\, ~\, \, {\mathcal 
C}^0\otimes_{\mathbb C}{\mathcal C}^0\,
\stackrel{\text{Id}\otimes {\mathcal D}_{\mathcal P}\oplus
{\mathcal D}_{\mathcal P}\otimes\text{Id}}{\longrightarrow}\,
(TX\otimes_{\mathbb C}K^{\otimes 2}_X)\oplus
(K^{\otimes 2}_X\otimes_{\mathbb C}TX)
$$
$$
\stackrel{{\mathcal D}_{\mathcal P}\otimes\text{Id}-
\text{Id}\otimes {\mathcal D}_{\mathcal P}}{\longrightarrow}\,
{\mathcal C}^1\otimes_{\mathbb C}{\mathcal C}^1\, .
$$
Also, consider the complex
$$
K_X[1]\, \,~\, ~\, {~:~}\,~\, ~\, \, 0\, \longrightarrow\, K_X\, .
$$
The natural contraction of $TX$ with $K_X$ defines a
homomorphism of complexes
$$
{\mathcal C}^\bullet\otimes_{\mathbb C}{\mathcal C}^\bullet\, 
\longrightarrow\, K_X[1]\, .
$$
Using it, we have homomorphisms
$$
{\mathbb H}^1(X,\, {\mathcal C}^\bullet)\otimes_{\mathbb C}
{\mathbb H}^1(X,\, {\mathcal C}^\bullet)\,\longrightarrow\,
{\mathbb H}^2(X,\, {\mathcal C}^\bullet\otimes_{\mathbb C}
{\mathcal C}^\bullet)\,\longrightarrow\,
{\mathbb H}^2(X,\, K_X[1])\, =\, H^1(X,\, K_X)\, =\,
\mathbb C\, .
$$
The resulting pairing on $T_{\mathcal P}{\mathcal P}_0(X)$
(see Eq. \eqref{sle12}) coincides with the natural symplectic
form on ${\mathcal P}_0(X)$.

\medskip
\noindent
\textbf{Acknowledgements.} We thank the referee for helpful
comments.


\end{document}